\documentclass{amsart}%
\usepackage[latin9]{inputenc}
\usepackage{amstext}
\usepackage{amsthm}
\usepackage{amssymb}
\usepackage{amsfonts}
\usepackage[dvips]{graphicx}
\usepackage{amsmath}%
\providecommand{\U}[1]{\protect\rule{.1in}{.1in}}
\newtheorem{theorem}{Theorem}
\theoremstyle{plain}

\newtheorem{corollary}{Corollary}

\newtheorem{example}{Example}

\newtheorem{lemma}{Lemma}

\newtheorem{problem}{Problem}
\newtheorem{proposition}{Proposition}
\newtheorem{remark}{Remark}

\numberwithin{equation}{section}
\begin{document}
\title{Old and new on the 3-convex functions}
\author{Dan-\c{S}tefan Marinescu}
\address{National College "Iancu de Hunedoara", Hunedoara, Romania}
\email{marinescuds@gmail.com}
\author{Constantin P. Niculescu}
\address{Department of Mathematics, University of Craiova, Craiova 200585, Romania}
\email{constantin.p.niculescu@gmail.com}
\date{May 7, 2023}

\begin{abstract}
The present paper aims to survey known results and to point out the wealth of
rather important open problems that are out there.

\end{abstract}
\maketitle

\section{Introduction}

Higher order convexity was introduced by Hopf \cite{Hopf} and Popoviciu
\cite{Pop34}, \cite{Pop1944}, who defined it in terms of divided differences
of a function. Assuming $f$ a real-valued function defined on a real interval
$I,$ the divided differences of order $0,1,\ldots,n$ associated to a family
$x_{0},x_{1},\ldots,x_{n}$ of $n+1$ distinct points are respectively defined
by the formulas
\begin{align*}
\lbrack x_{0};f]  &  =f(x_{0})\\
\lbrack x_{0},x_{1};f]  &  =\frac{f(x_{1})-f(x_{0})}{x_{1}-x_{0}}\\
&  ...\\
\lbrack x_{0},x_{1},...,x_{n};f]  &  =\frac{[x_{1},x_{2},...,x_{n}%
;f]-[x_{0},x_{1},...,x_{n-1};f]}{x_{n}-x_{0}}\\
&  =%
{\displaystyle\sum\nolimits_{j=0}^{n}}
\frac{f(x_{j})}{\prod\nolimits_{k\neq j}\left(  x_{j}-x_{k}\right)  }.
\end{align*}

Notice that all these divided differences are invariant under the permutation
of points $x_{0},x_{1},...,x_{n}.$ As a consequence, we may always assume that
$x_{0}<x_{1}<\cdots<x_{n}.$

A function $f$ is called $n$-\emph{convex }(respectively\emph{ }%
$n$-\emph{concave}) if all divided differences $[x_{0},x_{1},\ldots,x_{n};f]$
are nonnegative (respectively nonpositive). In particular,

\begin{itemize}
\item the convex functions of order 0 are precisely the nonnegative functions;

\item the convex functions of order 1 are the nondecreasing functions;

\item the convex functions of order 2 are nothing but the usual convex
functions since in this case for all $x_{0}<x_{1}<x_{2}$ in $I,$%
\[
\lbrack x_{0},x_{1},x_{2};f]=\frac{\frac{f(x_{0})-f(x_{1})}{x_{0}-x_{1}}%
-\frac{f(x_{1})-f(x_{2})}{x_{1}-x_{2}}}{x_{0}-x_{2}}\geq0,
\]
that is, $\left(  x_{2}-x_{0}\right)  f(x_{1})\leq\left(  x_{2}-x_{1}\right)
f(x_{0})+\left(  x_{1}-x_{0}\right)  f(x_{2}).$
\end{itemize}

While the properties of the above three classes of $n$-convex functions are
well understood, only few relevant results are known in the case where
$n\geq3.$ Besides the work of Popoviciu (see \cite{Pop34}, \cite{Pop35},
\cite{Pop1944}, \cite{Pop1946} and \cite{Pop69}) we should mention here the
contribution of Bennett \cite{Ben2010}, Bessenyei and Páles \cite{BP2002},
Boas and Widder \cite{BW1940}, Bojani\'{c} and Roulier \cite{BR}, Brady
\cite{Br}, Bullen \cite{Bul1971}, \cite{Bul1973}, Kuczma \cite{Kuc2009},
Marinescu and Monea \cite{MM}, Pecari\'{c} and his collaborators
\cite{KPP2013}, \cite{KPP2019}, \cite{MPF}, \cite{PPT}, Rajba \cite{Rajba},
Szostok \cite{Sz2021} and Wasowicz \cite{W}.

This paper is aimed to present an overview of the present state of art
concerning the $3$-convex functions, to add some new results, and to single
out some open problems which seem of interest. The reason to restrict
ourselves to this particular case is two-fold: it offers a convenient
framework to illustrate the richness of the class of 3-convex functions and
also a context that keeps the technical aspects still simple and intuitive.

For the convenience of the reader, some very basic facts are recalled in
Section 2.

Section 3 is mostly dedicated to the identity of 3-convex functions with the
functions having positive differences up to order 3. See Theorem
\ref{thm3conv} below. This result, already known to Popoviciu, was rarely
mentioned by the various books dedicated to convex functions, except that by
Kuczma \cite{Kuc2009}, that also includes a full proof. Our approach here
combines results due to Hopf \cite{Hopf}, Popoviciu \cite{Pop34}, Boas and
Widder \cite{BW1940} and Bennett \cite{Ben2010} (to cite them according to
their apparition). An important role is played by the fact that on an open
interval the 3-convex functions are precisely the differentiable functions
whose derivatives are convex functions. This offers the possibility to deduce
results for the 3-convex functions from know results for the usual convex
functions and vice-versa. See Theorem 3 and Theorem 4 and the comments that
accompany them in Section 4.

Section 5 is devoted to an overview of the Hermite-Hadamard inequality in the
context of continuous 3-convex functions. The central result is the remarkable
extension obtained by Bessenyei and Páles \cite{BP2002}, \cite{BP2010}, that
covers the general case of Borel probability measures on a compact interval
$[a,b]$. The Hermite-Hadamard inequality for the usual convex functions is
centered around the role played by the barycenter $p$ and the extremal points.
As $f(p)=\delta_{p}(f),$ in the $3$-convexity variant, the role of $\delta
_{p}$ is taken by a discrete probability measure supported at two points, one
of them being inside the interval. Naturally, this raises the interesting
problem how looks the analog of Choquet's theory in the framework of
$3$-convex functions. Several open problems in this connection are mentioned
at the end of Section 5.

Section 6 present a radiography of a recent result due Ressel \cite{Res},
concerning the connection between the functions which are continuous,
nondecreasing, concave and 3-convex and the Hornich-Hlawka inequality. See
Theorem . We show that this result is actually the juxtaposition of two
distinct results covering complementary domains of the variables, one
involving the properties of continuity and 3-convexity, while the other
involving only the properties of monotonicity and concavity. See Lemma
\ref{lem4a}, and Lemma \ref{lem4bcd} below.

The paper ends with a section discussing the extension of the entire theory to
the case of functions taking values in an ordered Banach space.

\section{Preliminaries}

A function $f$ defined on an interval $I$ is $3$-\emph{convex} if for every
quadruple $x_{0}<x_{1}<x_{2}<x_{3}$ of elements in $I,$%
\begin{multline*}
\lbrack x_{0},x_{1},x_{2},x_{3};f]=\frac{f(x_{0})}{(x_{0}-x_{1})(x_{0}%
-x_{2})(x_{0}-x_{3})}-\frac{f(x_{1})}{(x_{0}-x_{1})(x_{1}-x_{2})(x_{1}-x_{3}%
)}\\
+\frac{f(x_{2})}{(x_{0}-x_{2})(x_{1}-x_{2})(x_{2}-x_{3})}-\frac{f(x_{3}%
)}{(x_{0}-x_{3})(x_{1}-x_{3})(x_{2}-x_{3})}\geq0,
\end{multline*}
equivalently,%
\begin{align}
&  (x_{2}-x_{0})(x_{3}-x_{0})(x_{3}-x_{2})f(x_{1})+(x_{1}-x_{0})(x_{2}%
-x_{0})(x_{2}-x_{1})f(x_{3})\label{eq3conv}\\
&  \geq(x_{2}-x_{1})(x_{3}-x_{1})(x_{3}-x_{2})f(x_{0})+(x_{1}-x_{0}%
)(x_{3}-x_{0})(x_{3}-x_{1})f(x_{2}).\nonumber
\end{align}
When the points $x_{0},x_{1},x_{2},x_{3}$ are equidistant, that is, when
$x_{1}=x_{0}+h,$ $x_{2}=x_{0}+2h,$ $x_{3}=x_{0}+3h$ for some $h>0,$ the last
inequality becomes%
\[
f(x_{0}+3h)-3f(x_{0}+2h)+3f(x_{0}+h)-f(x_{0})\geq0,
\]
equivalently,%
\begin{equation}
f(x_{0})+3f\left(  \frac{x_{0}+2x_{3}}{3}\right)  \leq3f\left(  \frac
{2x_{0}+x_{3}}{3}\right)  +f(x_{3}). \label{eq3convJ}%
\end{equation}

If $f$ is $n$-times differentiable, then a repeated application of Lagrange's
mean value theorem yields the existence of a point $\xi\in\left(  \min
_{k}x_{k},\max_{k}x_{k}\right)  $ such that
\[
\lbrack x_{0},x_{1},...,x_{n};f]=\frac{f^{(n)}(\xi)}{n!}.
\]

As a consequence, one obtains the sufficiency part of the following practical
criterion of $n$-convex.

\begin{lemma}
\label{lemH}Suppose that $f$ is a continuous function defined on an interval
$I$ which is $3$-times differentiable on the interior of $I.$ Then $f$ is
$3$-convex if and only if its third derivative is nonnegative.
\end{lemma}

The necessity part is also immediate by using the standard formulas for
derivatives via iterated differences,%
\[
f^{\prime\prime\prime}(x_{0})=\lim_{h\rightarrow0+}\frac{f(x_{0}%
+3h)-3f(x_{0}+2h)+3f(x_{0}+h)-f(x_{0})}{h^{3}}.
\]

In connection with Lemma \ref{lemH}, it is worth mentioning a result due to
Hopf \cite{Hopf}, p. 24, and Popoviciu \cite{Pop34}, p. 48, which asserts that
every $3$-convex function $f$ defined on an open interval is differentiable
and $f^{\prime}$ is convex. This can be easily turned into a characterization
of 3-convexity in the framework of continuous functions. See Theorem
\ref{thm3conv} below.

Lemma \ref{lemH} allows us to notice the existence of a large variety of
$3$-convex functions:
\begin{gather*}
x/(x+1),~1-e^{-\alpha x}\text{ (for }\alpha>0),\text{ }\\
\log(1+x),\text{ }-x\log x\,,\ (x-1)/\log x,\text{ }x^{\alpha}\text{
}(\text{for }\alpha\in(0,1]\cup\lbrack2,\infty)),\\
-x^{2}+\sqrt{x},~\sinh,~\cosh,~-\log\left(  \Gamma(x)\right)  .
\end{gather*}
Also the primitive of any continuous convex function is a $3$-convex function.

The function $1-\left(  x-3\right)  +\frac{\left(  x-3\right)  ^{3}}{6}$ is
continuous and $3$-convex on $\mathbb{R}_{+}$ but not $n$-convex for any
$n\in\left\{  0,1,2\right\}  .$

The polynomials with positive coefficients and the exponential are $n$-convex
for every $n\geq0.$

\begin{remark}
All polynomials of degree less than or equal to $2$ are both $3$-convex and
$3$-concave. These functions together with the finite sums $\sum
\nolimits_{i=1}^{m}c_{i}\left(  \left(  x-a_{i}\right)  _{+}\right)  ^{2}$
with positive coefficients represent the building blocks of any $3$-convex
function. See Popoviciu \emph{\cite{Pop1944}}, pp. $29$-$30$\ $($and also
\emph{\cite{BR}} and \emph{\cite{Toader}}$)$.
\end{remark}

The continuous 3-convex functions on interval $I$ constitute a convex cone in
the vector space $C(I),$ of all continuous functions on $I.$ This cone is
closed under convolution, but not under usual product.

The extremal properties of $3$-convex functions differ from those of convex functions.

The maximum (or minimum) of two $3$-convex function is not necessarily a
$3$-convex function; consider the case of the functions $-x$ and $x.$ Also, an
interior critical point of a $3$-convex function is not necessarily a point of
minimum. It can be a point of inflection (the case of the cubic function
$x^{3})$ or a point of maximum (the case of the function $-x^{2}+\sqrt{x}).$

The following approximation theorem due to Popoviciu\emph{ }\cite{Pop35}\emph{
}(see also \cite{Gal2008}, Theorem 1.3.1 $(i)$, p. 20) allows us to reduce the
reasoning with $n$-convex functions to the case where they are also differentiable.

\begin{theorem}
$($Popoviciu's approximation theorem$)$\label{thm_Pop35} If a continuous
function $f:[0,1]\rightarrow\mathbb{R}$ is $k$-convex, then so are the
Bernstein polynomials associated to it,
\[
B_{n}(f)(x)=\sum_{i=0}^{n}\binom{n}{i}x^{i}(1-x)^{n-i}f\left(  \frac{i}%
{n}\right)  .
\]
Moreover, by the well-known property of simultaneous uniform approximation of
a function and its derivatives by the Bernstein polynomials and their
derivatives, it follows that $B_{n}(f)$ and any derivative (of any order) of
it, converge uniformly to $f$ and to its derivatives, correspondingly.
\end{theorem}

Using a change of variable, one can easily see that the approximation theorem
extends to functions defined on compact intervals $[a,b]$ with $a<b.$

\begin{corollary}
\label{cor_comp}If $f:\mathbb{R}_{+}\rightarrow\mathbb{R}_{+}$ is a continuous
$3$-convex function which is also nondecreasing and concave, then the same
properties hold for $f^{\alpha}$ if $\alpha\in(0,1].$
\end{corollary}

\begin{proof}
According to Theorem \ref{thm_Pop35}, we may reduce the proof to the case
where the involved function is of class $C^{3}$, in which case the conclusion
follows from Lemma \ref{lemH}.
\end{proof}

\begin{remark}
An important class of functions $f:\mathbb{R}_{+}\rightarrow\mathbb{R}_{+}$
which are continuous, nondecreasing concave and $3$-convex is that of
Bernstein functions. Recall that a function $\mathbb{R}_{+}\rightarrow
\mathbb{R}_{+}$ is called Bernstein if it is continuous on $\mathbb{R}_{+},$
indefinitely differentiable on $(0,\infty)$ and
\[
(-1)^{n+1}f^{(n)}\geq0\text{\quad for all }n\geq1.
\]
Their theory is exposed in the monograph of Schilling, Song, and Vondra\v{c}ek
\emph{\cite{SSV}}. According to Corollary 3.8, p. 28, in this monograph, the
composition of two Bernstein functions is also a Bernstein function $($which
is not the case for the class of $3$-convex functions$)$.
\end{remark}

\section{Some characterizations of 3-convexity}

The \emph{difference operator} $\Delta_{h}$ $($of step size $h\geq0)$
associates to each function $f$ defined on an interval $I$ the function
$\Delta_{h}f$ defined by%
\[
\left(  \Delta_{h}f\right)  (x)=f(x+h)-f(x),
\]
for all $x$ such that the right-hand side formula makes sense. Notice that no
restrictions are necessary if $I=\mathbb{R}^{+}$ or $I=\mathbb{R}.$

The difference operators are linear and commute to each other,%
\[
\Delta_{h_{1}}\Delta_{h_{2}}=\Delta_{h_{2}}\Delta_{h_{1}}.
\]

They also verify the following property of invariance under translation:%
\[
\Delta_{h}\left(  f\circ T_{a}\right)  =\left(  \Delta_{h}f\right)  \circ
T_{a},
\]
where $T_{a}$ is the translation defined by the formula $T_{a}(x)=x+a.$

\begin{lemma}
\label{lem_iter}If $n$ is a positive integer, then the following formula
holds:
\[
\Delta_{h_{1}}\Delta_{h_{2}}\cdots\Delta_{h_{n}}f(x)=%
{\displaystyle\sum\limits_{\varepsilon_{1},...,\varepsilon_{n}\in\left\{
0,1\right\}  }}
\left(  -1\right)  ^{n-\left(  \varepsilon_{1}+\cdots+\varepsilon_{n}\right)
}f\left(  x+\varepsilon_{1}h_{1}+\cdots+\varepsilon_{n}h_{n}\right)  .
\]

\end{lemma}

The proof is immediate, by mathematical induction.

A function $f$ has \emph{positive differences of order} $n\geq1$ if
\[
\Delta_{h_{1}}\Delta_{h_{2}}\cdots\Delta_{h_{n}}f(x)\geq0,
\]
whenever the left-hand size is well defined. To outline a parallel to
$n$-convexity, we say that a function $f$ has positive differences of order 0
if $f\geq0.$

Notice that a function has positive differences of first order if it is
nondecreasing and it has positive differences of second order if it is convex
(a simple exercise left to the reader). As was noticed by Popoviciu
\cite{Pop1944} (at the beginning of Section 24, p. 49), this remark works in
the general case of $n$-convex continuous functions. A detailed proof can be
found in the book of Kuczma \cite{Kuc2009}; see Theorem 15.6.1, p. 440. More
comments are available in \cite{NS2023}. The next result concerns the case
$3$-convex functions.

\begin{theorem}
\label{thm3conv}If $f:[0,A]\rightarrow\mathbb{R}$ is a continuous function
then the following assertions are equivalent:

$(i)$ $f$ is $3$-convex$;$

$(ii)$ $\Delta_{h}\Delta_{h}\Delta_{h}f(x)=f(x+3h)-3f(x+2h)+3f(x+h)-f(x)\geq0$
for all $x\in\lbrack0,A)$ and $h>0$ such that $x+3h\leq A.$

$(iii)$ $f$ has positive differences of order $3$, that is, it verifies the
inequality
\begin{multline*}
\Delta_{x}\Delta_{y}\Delta_{z}f(t)=f\left(  x+t\right)  +f\left(  y+t\right)
+f\left(  z+t\right)  +f\left(  x+y+z+t\right) \\
-f\left(  x+y+t\right)  -f\left(  y+z+t\right)  -f\left(  z+x+t\right)
-f(t)\geq0
\end{multline*}
whenever $x,y,z,t\in\lbrack0,A]$ and $x+y+z+t\leq A;$

$(iv)~f$ is differentiable on $(0,A)$ and its derivative $f^{\prime}$ is a
convex function.
\end{theorem}

\begin{proof}
The implication $(i)\Longrightarrow(ii)$ is straightforward. The fact that
$(i)\Longleftrightarrow(ii)$ under the presence of continuity is stated by
Popoviciu in his book \cite{Pop1944} (at the beginning of Section 24, p. 49).

The implication $(i)\Longrightarrow(iv)$ was noticed both by Hopf \cite{Hopf}
and Popoviciu \cite{Pop34} (in the general case of $n$-convex functions).

The implications $(ii)\Longrightarrow(iii)\Longrightarrow(iv)$ are covered by
the paper of Boas and Widder \cite{BW1940} (see Lemma 1 and the Theorem, p. 497).

The implication $(iv)\Longrightarrow(i)$ was noticed by Bennett
(\cite{Ben2010}, Proposition 1), who used the identity%
\begin{align*}
\lbrack a,b,c,d;f]  &  =\frac{1}{\left(  b-a\right)  \left(  c-a\right)
\left(  d-a\right)  }\int_{a}^{b}f^{\prime}(t)dt\\
&  -\frac{c+d-a-b}{\left(  c-a\right)  \left(  c-b\right)  \left(  d-a\right)
\left(  d-b\right)  }\int_{b}^{c}f^{\prime}(t)dt\\
&  +\frac{1}{\left(  d-a\right)  \left(  d-b\right)  \left(  d-c\right)  }%
\int_{c}^{d}f^{\prime}(t)dt,
\end{align*}
for all \thinspace$a<b<c<d.$
\end{proof}

It is worth mentioning that the property of $3$-convex functions of having
positive differences up to order 3 can be also deduced by adapting the
argument used by Popoviciu \cite{Pop1946} for a weaker variant of it.

We shall need the following special case of the Hardy-Littlewood-Pólya
inequality of majorization (see \cite{NP2018}, Theorem 4.1.3, p. 186):

\begin{lemma}
\label{lem_maj}If\emph{ }$g:\left[  a,b\right]  \rightarrow%
\mathbb{R}
$ is a convex function and\emph{ }$c$ and $d$ are two points in $\left[
a,b\right]  $ such that $a+b=c+d$, then%
\[
g(c)+g(d)\leq g(a)+g(b).
\]

\end{lemma}

Popoviciu's alternative argument for the implication $(i)\Longrightarrow(iii)$
in Theorem \ref{thm3conv} works as follows: According to Popoviciu's
approximation theorem, we may restrict to the case where $f$ is a $3$-convex
function of class $C^{\infty}.$ Then fix arbitrarily $y,z,t$ in $[0,A]$ such
that $y+z+t<A$ and consider the function
\begin{align*}
F(x)  &  =f\left(  x+t\right)  +f\left(  y+t\right)  +f\left(  z+t\right)
+f\left(  x+y+z+t\right) \\
&  -f\left(  x+y+t\right)  -f\left(  y+z+t\right)  -f\left(  z+x+t\right)
-f(t)
\end{align*}
defined on the interval $[0,A-y-z-t].$ This function is also of class
$C^{\infty}$ and
\[
F^{^{\prime}}\left(  x\right)  =f^{^{\prime}}\left(  x+t\right)  +f^{^{\prime
}}\left(  x+y+z+t\right)  -f^{^{\prime}}\left(  x+y+t\right)  -f^{^{\prime}%
}\left(  z+x+t\right)  .
\]

From Lemma \ref{lemH} we infer that $F^{\prime}$ is a convex function, so that
$F^{\prime}\geq0$ according to Lemma \ref{lem_maj}. Therefore
\[
F(x)\geq F(0)=0
\]
for all $x,$ a fact which is equivalent to the assertion \ $(iii).$ The proof
is done.

\smallskip

The connection of Theorem \ref{thm3conv}\emph{ }$(iii)$ with the
Hornich-Hlawka inequality will be discussed in Section \ref{sectionHornich}.
The next section is devoted to the applications of Theorem \ref{thm3conv}%
\emph{ }$(iv).$

\section{\smallskip Applications of Theorem \ref{thm3conv}\emph{ }$(iv)$}

Theorem \ref{thm3conv}\emph{ }$(iv)$ easily allows us to deduce results for
the usual convex functions from those for the 3-convex functions (and
vice-versa). Two examples are exhibited below. The first one is a refinement
of the Jensen inequality.

\begin{theorem}
\label{thm3=>2}A continuous function $f$ defined on an interval $I$ is convex
if and only if%
\[
\frac{1}{\left\vert J\right\vert }\int_{J}f(x)\mathrm{d}x\geq\frac
{1}{\left\vert K\right\vert }\int_{K}f(x)\mathrm{d}x,
\]
whenever $K\subset J$ are two compact subintervals of $I$ with the same midpoint.
\end{theorem}

\begin{proof}
Suppose that $f$ is convex and that the two intervals under attention are
$J=[a,b]$ and $K=[a+\varepsilon,b-\varepsilon]$ (where $\varepsilon
\in(0,(b-a)/2))$. Every primitive $F$ of $f$ is a $3$-convex function of class
$C^{1}$ so, according to the formula (\ref{eq3conv}) (applied to the points
$a<a+\varepsilon<b-\varepsilon<b),$ it verifies the inequality
\[
F(b)-F(a)\geq\frac{b-a}{b-a-2\varepsilon}\left[  F(b-\varepsilon
)-F(a+\varepsilon)\right]  .
\]
Thus%
\begin{align*}
\frac{1}{b-a}\int_{a}^{b}f(x)\mathrm{d}x  &  =\frac{1}{b-a}\int_{a}%
^{b}F^{\prime}(x)\mathrm{d}x=\frac{F(b)-F(a)}{b-a}\\
&  \geq\frac{1}{b-a-2\varepsilon}\left[  F(b-\varepsilon)-F(a+\varepsilon
)\right] \\
&  =\frac{1}{b-a-2\varepsilon}\int_{a+\varepsilon}^{b-\varepsilon
}f(x)\mathrm{d}x
\end{align*}
and the proof of the necessity part is done.

For the sufficiency part, notice that $f$ verifies the condition%
\[
\frac{1}{\left\vert I_{n}\right\vert }\int_{I_{n}}f(x)dx\searrow f\left(
\frac{a+b}{2}\right)  ,
\]
whenever $I_{0}=[a,b]\supset I_{1}\supset I_{2}\supset\cdots$ is a sequence of
nested compact subintervals of $I$ that shrink to $(a+b)/2,$ supposed to be
their common midpoint. Therefore $f$ is a continuous function such that%
\[
\frac{1}{b-a}\int_{a}^{b}f(x)\mathrm{d}x\geq f\left(  \frac{a+b}{2}\right)
\]
whenever $a<b$ in $I.$ As it is well known, this property implies that $f$ is
a convex function. See \cite{NP2018}, Exercise 1, p. 63.
\end{proof}

\begin{remark}
\label{rem3=>2}In the same manner one can prove that a continuous function $f$
defined on an interval $I$ is convex if and only if%
\[
\frac{1}{2\varepsilon}\int_{a}^{a+\varepsilon}f(x)\mathrm{d}x+\frac
{1}{2\varepsilon}\int_{b-\varepsilon}^{b}f(x)\mathrm{d}x\geq\frac{1}{b-a}%
\int_{a}^{b}f(x)\mathrm{d}x,
\]
whenever $a<b$ in $I$ and $\varepsilon\in(0,(b-a)/2).$ This represents a
refinement of the right-hand side of the Hermite-Hadamard inequality $($see
\emph{\cite{NP2018}}, Section $1.10$, pp. $59-64)$.
\end{remark}

Combining Theorem \ref{thm3=>2} and Remark \ref{rem3=>2} one obtains double
inequalities such as%
\begin{multline*}
\frac{4}{b-a}\int_{a}^{\left(  3a+b\right)  /4}f(x)\mathrm{d}x+\frac{4}%
{b-a}\int_{(a+3b)/4}^{b}f(x)\mathrm{d}x\\
\geq\frac{1}{b-a}\int_{a}^{b}f(x)\mathrm{d}x\geq\frac{3}{b-a}\int_{\left(
2a+b\right)  /3}^{(a+2b)/3}f(x)\mathrm{d}x,
\end{multline*}
for every convex function $f:[a,b]\rightarrow\mathbb{R}.$ Continuity of $f$ is
not necessary. See \cite{NP2018}, Proposition 1.1.3, p. 3.

Similarly, one can use the theory of convex functions to characterize the
3-convex functions.

\begin{theorem}
\label{thm3HH}Suppose that $f:I\rightarrow\mathbb{R}$ is a function continuous
on $I$ and of class $C^{1}$ on the interior of $I.$ Then the following
assertions are equivalent:

$(i)$ $f$ is $3$-convex;

$(ii)$ $f$ verifies the inequality
\begin{equation}
f^{\prime}\left(  \frac{a+b}{2}\right)  \leq\frac{f(b)-f(a)}{b-a}, \label{eqJ}%
\end{equation}
for all points $a<b$ in $\operatorname{int}I;$

$(ii)$ $f$ verifies the inequality
\begin{equation}
\frac{f(b)-f(a)}{b-a}\leq\frac{1}{2}\left(  \frac{f^{\prime}(a)+f^{\prime}%
(b)}{2}+f^{\prime}\left(  \frac{a+b}{2}\right)  \right)  , \label{eqHH}%
\end{equation}
for all points $a<b$ in $\operatorname{int}I$.
\end{theorem}

\begin{proof}
$(i)\Longrightarrow(ii)$ According to Theorem \ref{thm3conv}, $f$ is
differentiable on $(a,b)$ and $f^{\prime}$ is a convex function on this
interval. Taking into account the Jensen inequality (see \cite{NP2018},
Corollary 1.7.4, p. 43), for every $\varepsilon\in(0,(b-a)/2),$ we have%
\[
f^{\prime}\left(  \frac{a+b}{2}\right)  \leq\frac{1}{b-a}\int_{a+\varepsilon
}^{b-\varepsilon}f^{\prime}(x)\mathrm{d}x=\frac{f(b-\varepsilon
)-f(a+\varepsilon)}{b-a}%
\]
and the inequality (\ref{eqJ})\ follows by passing to the limit as
$\varepsilon\rightarrow0+.$

$(ii)\Longrightarrow(i)$ Indeed, the inequality (\ref{eqJ}) assures the
fulfillment of the inequality Jensen inequality by the function $f^{\prime
\prime}$ on every compact interval included in $\operatorname{int}I,$ which is
known to imply the convexity of $f^{\prime}.$ See \cite{NP2018}, Exercise 1,
p. 63.

$(i)\Longrightarrow(iii)$ The proof of inequality (\ref{eqHH}) can be done in
the same manner by using Remark 1.10.5, p. 61, in \cite{NP2018}. The
implication $(iii)\Longrightarrow(i)$ follows from Exercise 2, p. 63, loc. cit.)
\end{proof}

\begin{example}
\label{ex3HH_1}Applying Theorem \emph{\ref{thm3HH}} in the case of the
function $f(x)=\log(1+x),$ $x\geq0,$ one obtains the double inequality%
\[
\frac{b-a}{1+\frac{a+b}{2}}\leq\log\frac{1+b}{1+a}\leq\left(  \frac{b-a}%
{4}\right)  \left(  \frac{1}{1+a}+\frac{1}{1+b}+\frac{2}{1+\frac{a+b}{2}%
}\right)  ,
\]
valid for all $0\leq a\leq b.$ This provides a rational estimate of
$\log(1+x),$ better than the estimate offered by Maclaurin's expansion.
\end{example}

\begin{example}
If $f:[a,b]\rightarrow\mathbb{R}$ is a $3$-times differentiable function with
$M=\sup_{x\in\lbrack a,b]}f^{\prime\prime\prime}(x)<\infty,$ then $Mx^{3}/6-f$
is a continuous $3$-convex function. For example, in the case of the sine
function, this works for $M=1,$ which implies that%
\begin{multline*}
\frac{\left(  a+b\right)  ^{2}}{2}-\cos\frac{a+b}{2}\leq\frac{a^{2}+ab+b^{2}%
}{6}-\frac{\sin b-\sin a}{b-a}\\
\leq\frac{1}{4}\left(  \frac{a^{2}+ab+b^{2}}{2}-\cos a-\cos b-2\cos\left(
\frac{a+b}{2}\right)  \right)  .
\end{multline*}
Similarly, when $m=\inf_{x\in\lbrack a,b]}f^{\prime\prime\prime}(x)>-\infty,$
then the function $f-mx^{3}/6$ is $3$-convex and the conclusion of Theorem
\emph{\ref{thm3HH} }applies to it.
\end{example}

Another application of Theorem \ref{thm3conv} refers to the "support" of a
$3$-convex functions.

If $f$ is a continuous $3$-convex function defined on an interval $I$, then
$f^{\prime}$ is a convex function on $\operatorname{int}I$ and we can apply to
it the theory of subdifferentiability of convex functions. According to
\cite{NP2018}\emph{,} Theorem $1.6.2$, p. $36$, the subdifferential $\partial
f^{\prime}(a)$ $($of $f^{\prime}$ at a point $a$ interior to $I)$ equals the
interval $[f_{-}^{\prime\prime}(a),f_{+}^{\prime\prime}(a)]$ and
\[
f^{\prime}(x)\geq f^{\prime}(a)+y(x-a)\text{\quad for all }x\in
\operatorname{int}I\text{ and }y\in\partial f^{\prime}(a).
\]
Notice that $\partial f^{\prime}(a)=\left\{  f^{\prime\prime}(a)\right\}  $
when $f$ is twice differentiable at $a$.

Therefore%
\begin{equation}
f(x)\geq f(a)+(x-a)f^{\prime}(a)+\frac{\left(  x-a\right)  ^{2}}{2}%
f_{+}^{\prime\prime}(a)\text{\quad if }x\in I\text{ and }x\geq a
\label{tp-right}%
\end{equation}
and%
\begin{equation}
f(x)\leq f(a)+(x-a)f^{\prime}(a)+\frac{\left(  x-a\right)  ^{2}}{2}%
f_{-}^{\prime\prime}(a)\text{\quad if }x\in I\text{ and }x\leq a.
\label{tp-left}%
\end{equation}
As a consequence,%
\[
f(x)=\sup\left\{  f(a)+(x-a)f^{\prime}(a)+\frac{\left(  x-a\right)  ^{2}}%
{2}y:a\in\operatorname{int}I,\text{ }a<x,\text{ }y\in\partial f^{\prime
}(a)\right\}  \ \ \ \
\]
for all $x\in I$ different from the left endpoint of $I.$

Geometrically, $g=f(a)+(x-a)f^{\prime}(a)+\frac{\left(  x-a\right)  ^{2}}%
{2}f_{+}^{\prime\prime}(a)$ $($respectively $y=f(a)+(x-a)f^{\prime}%
(a)+\frac{\left(  x-a\right)  ^{2}}{2}f_{-}^{\prime\prime}(a))$ represents the
unique parabola tangent to the graph of $f$ at $a$ and which has the same
right-hand $($left-hand$)$ derivative of second order at $a.$ These "tight
tangent parabolas" represent analogs of the line supports from the theory of
convex functions. Notice that a tight tangent parabola at a point $a$ is above
the graph of $f$ on $I\cap(-\infty,a]$ and under the graph on $I\cap\lbrack
a,\infty).$

\begin{problem}
Suppose that $f$ is a function continuous on $I$ and differentiable on
$\operatorname{int}I,$ which admit at every point $a\in\operatorname{int}I$ a
tight tangent parabola. Is $f$ necessarily $3$-convex?
\end{problem}

The answer to this problem seems to be positive as suggests Bullen's analogue
for $3$--convex functions, of the familiar fact that the graph of a convex
function lies always beneath its chords.

\begin{theorem}
\label{thmBullen}A continuous function $f$ is $3$-convex on $[a,b]$ if and
only if for every quadratic function $Q$ that agrees with $f$ at $\alpha$,
$\beta$ and $\gamma$ $($where $a<\alpha<\beta<\gamma<b)$ we have
\[
Q(t)\geq f(t)\text{\quad if }x\in\lbrack a,\alpha]\cup\lbrack\beta,\gamma]
\]
and%
\[
f(t)\geq Q(t)\text{\quad if }x\in\lbrack\alpha,\beta]\cup\lbrack\gamma,b].
\]

\end{theorem}

See \cite{Bul1971}, Theorem $5$, p. $85$ and also Theorem $10$, p. $88$.

The existence of tight tangent parabolas raises naturally the problem of an
analogue of Fenchel duality in the case of continuous $3$-convex functions.

\begin{problem}
Does there exist an analogue of Fenchel conjugate in the case of continuous
$3$-convex functions?
\end{problem}

In connection with the last problem, notice that a function which is both
convex and $3$-convex, may have a Fenchel conjugate which is not 3-convex. See
the case of the exponential function, whose Fenchel conjugate is the function
$f^{\ast}(x)=x\log x-x$ for $x>0$ and $f^{\ast}(0)=0.$

\begin{remark}
\label{remWas}The support-type properties of $3$-convex functions were
investigated also by Wasowicz who proved the following result: If
$f:[a,b]\rightarrow\mathbb{R}$ is a $3$-convex function, then for every point
$c\in(a,b)$ there exist quadratic functions $p$ and $q$ such that%
\[
p(a)=f(a),\text{ }p(c)=f(c)\text{ and }p\leq f\text{ on }[a,b]
\]
and%
\[
q(c)=f(c),\text{ }p(b)=f(b)\text{ and }q\geq f\text{ on }[a,b].
\]
See \emph{\cite{W}}, Corollaries $10$ and $11.$
\end{remark}

Another application of Theorem \ref{thm3conv}\emph{ }$(iv)$ is provided by the
inequality (\ref{ineq3HHF}), proved in the next section.

\section{The Hermite-Hadamard inequality in the context of 3-convex functions}

According to Choquet's theory, the meaning of the Hermite-Hadamard inequality
for convex functions on intervals is that of a double estimate for the
integral mean of every convex function $f:[a,b]\rightarrow\mathbb{R}$ with
respect to a Borel probability measure $\mu$ on $[a,b]$, precisely, \emph{\ }
\begin{equation}
\,f\left(  \operatorname*{bar}(\mu)\right)  \leq\,\int_{a}^{b}\,f(x)\mathrm{d}%
\mu\leq\frac{b-\operatorname*{bar}(\mu)}{b-a}\cdot f(a)+\frac
{\operatorname*{bar}(\mu)-a}{b-a}\cdot f(b)\,. \label{HH}%
\end{equation}

Here $\operatorname*{bar}(\mu)$ represents the barycenter of $\mu,$ that is,
the unique point $p$ in $[a,b]$ such that
\begin{equation}
f(p)=\int_{a}^{b}\,f(x)\,\mathrm{d}\mu(x)\, \label{bar}%
\end{equation}
for every continuous affine function $f:[a,b]\rightarrow\mathbb{R}$. One can
easily check that $\operatorname*{bar}(\mu)=\int_{a}^{b}x\,\mathrm{d}%
\mu(x)\,,$ the moment of the first order of $\mu.$

See \cite{NP2018} (and also \cite{N2002} and \cite{NP2003}).

When $\mu$ is the an absolutely continuous probability measure of the form
$wdx,$ where the weight $w\geq0$ is continuous and symmetric about the
vertical line $x=(a+b)/2,$ that is,%
\[
w(x)=w(a+b-x)\text{\quad for all }x\in\lbrack a,b],
\]
then $\operatorname*{bar}(\mu)=\left(  a+b\right)  /2$ and the inequality
(\ref{HH}) becomes%
\begin{equation}
\,f\left(  \frac{a+b}{2}\right)  \leq\,\int_{a}^{b}\,f(x)w(x)\mathrm{d}%
x\leq\frac{f(a)+f(b)}{2}\,, \label{FHH}%
\end{equation}
This is Fejér's variant of the classical Hermite-Hadamard inequality, also
known as the Hermite-Hadamard-Fejér inequality. See \cite{NP2018}, the remark
after Exercise $6$, p. $64.$

\begin{remark}
The inequality \emph{(\ref{HH})} also works outside the framework of
probability measure, for examples for the so called Hermite-Hadamard measures,
an example being $3(x^{2}-1/6)\mathrm{d}x$ on $[-1,1].$ For details, see
\emph{\cite{FN2007}} and \emph{\cite{NP2018}}, Section \emph{7.5}, pp.
\emph{322-324}.
\end{remark}

Theorem \ref{thm3conv} $(iv)$ allows us to derive from the
Hermite-Hadamard-Fejér inequality some consequences for the 3-convex functions.

For this, consider the case of a differentiable $3$-convex function
$f:[a,b]\rightarrow\mathbb{R}$ and of a continuous real weight $w$ which
admits a primitive $W\geq0,$ symmetric about the vertical line $x=(a+b)/2$.
Three such examples are: 1) $w(x)=a+b-2x$ on $[a,b]$ (with the primitive
$W(x)=(x-a)(b-x))$; 2) $2nx^{2n-1}$ on $[-a,a]$ (with the primitive $x^{2n});$
3) $\cos x$ on $[0,\pi]$ (with the primitive $\sin x).$

Then%
\begin{align*}
\int_{a}^{b}\,f(x)w(x)\mathrm{d}x  &  =fW|_{a}^{b}-\int_{a}^{b}\,f^{\prime
}(x)W(x)\mathrm{d}x\\
&  =f(b)W(b)-f(a)W(a)-\int_{a}^{b}\,f^{\prime}(x)W(x)\mathrm{d}x,
\end{align*}
and the Hermite-Hadamard-Fejér inequality leads to%
\begin{align}
&  -\frac{f^{\prime}(a)+f^{\prime}(b)}{2}\int_{a}^{b}\,W(x)\mathrm{d}%
x\label{ineq3HHF}\\
&  \leq\int_{a}^{b}\,f(x)w(x)\mathrm{d}x-\left(  f(b)W(b)-f(a)W(a)\right)
\leq-f^{\prime}\left(  \frac{a+b}{2}\right)  \int_{a}^{b}\,W(x)\mathrm{d}%
x.\nonumber
\end{align}

In the particular case where $w(x)=a+b-2x$ and $W(x)=(x-a)(b-x)$ we have
\[
\int_{a}^{b}W(x)\mathrm{d}x=\allowbreak\frac{1}{6}\left(  b-a\right)  ^{3}%
\]
and the inequality (\ref{ineq3HHF}) becomes%
\[
\frac{(b-a)^{3}}{6}f^{\prime}\left(  \frac{a+b}{2}\right)  \leq\int_{a}%
^{b}\,f(x)(2x-a-b)\mathrm{d}x\leq\frac{(b-a)^{3}}{12}\left(  f^{\prime
}(a)+f^{\prime}(b)\right)  .
\]

We pass now to the existence of an analogue of the Hermite-Hadamard inequality
for the $3$-convex functions.

As in the case of usual convex function it is useful to consider the following
$3$-convex ordering on the set $\operatorname*{Prob}([a,b],$ of all Borel
probability measures on $[a,b]:$%
\begin{multline*}
\nu\prec_{3cvx}\mu\text{ if and only if}\int_{a}^{b}f(x)\,\mathrm{d}\nu
(x)\leq\int_{a}^{b}f(x)\,\mathrm{d}\mu(x)\text{ }\\
\text{for all continuous and 3-convex functions }f:[a,b]\rightarrow\mathbb{R}.
\end{multline*}

Some important necessary and sufficient conditions for higher order convex
ordering are available in the papers of Denuit, Lefevre and Shaked
\cite{DLS98}, Rajba \cite{Rajba} and Szostok \cite{Sz2021}.

The relation $\prec_{3cvx}$is indeed a partial order relation. Clearly, it is
transitive and reflexive; the fact that $\nu\prec_{3cvx}\mu$ and $\mu
\prec_{3cvx}\nu$ imply $\mu=\nu$ comes from the fact that the linear space
generated by the continuous 3-convex functions is dense in $C([a,b])$.

\begin{remark}
For every $\mu\in\operatorname*{Prob}([a,b]$ one can choose a minimal Borel
probability measure $\lambda$ such that $\lambda\prec_{3cvx}\mu$ $($and the
same is true for the maximal measures$)$. Indeed, $\operatorname*{Prob}([a,b]$
can be identified with a weak star convex and compact subset of the dual space
of $C\left(  [a,b]\right)  ,$ precisely with $\left\{  x^{\ast}\in C\left(
[a,b]\right)  :x^{\ast}\geq0\text{ and }x^{\ast}(1)=1\right\}  ,$ so that
every net of measures minorizing $\mu$ admit a convergent subnet in the weak
star topology. Thus the existence of minimal Borel probability measures
majorized by $\mu$ follows from Zorn's lemma.
\end{remark}

\begin{remark}
Given a Borel probability measure $\mu$ on $[a,b]$ whose support includes more
than two points, \emph{no} Dirac measure $\delta_{p}$ can be found such that%
\[
\delta_{p}(f)=f(p)\leq\int_{a}^{b}f\left(  x\right)  \mathrm{d}\mu(x)
\]
for all continuous $3$-convex functions. Indeed, checking this for the
functions $\pm x$ and $\pm x^{2}$ we should have%
\[
p=\int_{a}^{b}x\mathrm{d}\mu(x)\text{ and }p^{2}=\int_{a}^{b}x^{2}%
\mathrm{d}\mu(x),
\]
which is not possible because the equality occurs in the Cauchy-Schwarz
inequality if and only if one the two functions involved is a scalar multiple
of the other.
\end{remark}

The last two remarks lead naturally to the problem of characterizing the
minimal Borel probability measures with respect to the ordering $\prec
_{3cvx}.$

\begin{problem}
Given a Borel probability measure $\mu$ on $[a,b],$ find numbers $\lambda
,\mu,\nu\in\lbrack0,1]$ such that for every continuous $3$-convex function
$f:[a,b]\rightarrow\mathbb{R}$ we have
\begin{equation}
\lambda f((1-\mu)a+\mu b)+(1-\lambda)f((1-\nu)a+\nu b)\leq\frac{1}{b-a}%
\int_{a}^{b}f(x)\mathrm{d}\mu(x). \tag{3J}\label{3J}%
\end{equation}

\end{problem}

A result due to Bessenyei and Páles \cite{BP2010}, Theorem 3.4, combined with
a careful inspection of the argument of Lemma 3.2 in \cite{Sz2021}, shows that
this problem has a unique solution, provided that the support of $\mu$
contains at least 3 points. \ \ In the particular case when $\mu$ equals
$\left(  1/(b-a)\right)  \mathrm{d}x,$ this solution corresponds to
\[
\lambda=1/4,\text{ }\mu=0\text{ and }\nu=2/3,
\]
and their result reads as follows:

\begin{theorem}
\label{thmBP}For a continuous function $f:I$ $\rightarrow$ $\mathbb{R}$, the
following statements are equivalent:

$(i)$ $f$\ is $3$-convex$;$

$(ii)~$for all $a,b\in I$ with $a<b,$%
\[
\frac{1}{4}f(a)+\frac{3}{4}f\left(  \frac{a+2b}{3}\right)  \leq\frac{1}%
{b-a}\int_{a}^{b}f(x)\mathrm{d}x;
\]

$(iii)~$for all $a,b\in I$ with $a<b,$%
\[
\frac{1}{b-a}\int_{a}^{b}f(x)\mathrm{d}x\leq\frac{3}{4}f\left(  \frac{2a+b}%
{3}\right)  +\frac{1}{4}f(b).
\]

\end{theorem}

In what follows we will refer to the extremal probability measures $\frac
{1}{4}\delta_{a}+\frac{3}{4}\delta_{\left(  a+3b\right)  /3}$ and $\frac{3}%
{4}\delta_{(2a+b)/3}+\frac{1}{4}\delta_{b}$ as the $3$-\emph{condensation} of
$(1/(b-a))\mathrm{d}x$\emph{ }and respectively the $3$-\emph{dispersion} of
$(1/(b-a))\mathrm{d}x.$The non symmetric form of these probability measures
seems to be a consequence of Bullen's Theorem \ref{thmBullen}.

Notice that%
\begin{align*}
f\left(  \frac{a+b}{2}\right)   &  \leq\frac{1}{4}f(a)+\frac{3}{4}f\left(
\frac{a+2b}{3}\right)  \leq\frac{1}{b-a}\int_{a}^{b}f(x)\mathrm{d}x\\
&  \leq\frac{3}{4}f\left(  \frac{2a+b}{3}\right)  +\frac{1}{4}f(b)\leq
\frac{f(a)+f(b)}{2}%
\end{align*}
in the case functions $f$ which are both convex and 3-convex.

Theorem \ref{thmBP} outlines the following property of rigidity of the
continuous $3$-convex functions:

\begin{remark}
If $f:[a,b]\rightarrow\mathbb{R}$ is a continuous $3$-convex function such
that $f(a)\geq0$ and $f\left(  \frac{a+2b}{3}\right)  \geq0$, then its
integral mean value is also greater than or equal to $0$. This imposes that
the values of $f$ in the interval $[\left(  a+2b\right)  /3,b]$ cannot be
"too" negative $($though it can be negative as shows the case of the function
$-x^{2}+\sqrt{x}$ defined on $[0,3/2]).$

Similarly, if $f\left(  \frac{2a+b}{3}\right)  \leq0$ and $f(b)\leq0,$ then
the integral mean value of $f$ is also less than or equal to $0$ $($and the
values of $f$ in the interval $[a,\left(  2a+b\right)  /3]$ cannot be "too"
positive$)$.
\end{remark}

There are numerous open problems related to Theorem \ref{thmBP} which seems of interest.

\begin{problem}
What is the statistical meaning of the $3$-condensation\emph{ }of
$(1/(b-a))\mathrm{d}x?$ The same in the case of the $3$-dispersion of
$(1/(b-a))\mathrm{d}x.$
\end{problem}

\begin{problem}
Is any Fejér analog of Theorem \emph{\ref{thmBP}}?
\end{problem}

\begin{problem}
Does the relation $\mu\prec_{3cvx}\nu$ admit a characterization à la Sherman
\emph{\cite{She1951},} when $\mu$ and $\nu$ are discrete probability measures?
\end{problem}

\begin{problem}
Does Theorem \emph{\ref{thmBP}} admit an extension to the context of signed
measures $($as is the case of the Hermite-Hadamard measures for convex
functions$)$?
\end{problem}

\section{3-convexity and the Hornich-Hlawka functional
inequality\label{sectionHornich}}

A straightforward consequence of Theorem \ref{thm3conv} is the fact that every
nonnegative and continuous $3$-convex function verifies the Hornich-Hlawka
functional inequality:

\begin{proposition}
\label{prop3conv}If $f:\left[  0,A\right]  \rightarrow\mathbb{R}$ is a
continuous $3$-convex function, then%
\begin{equation}
f\left(  x\right)  +f\left(  y\right)  +f\left(  z\right)  +f\left(
x+y+z\right)  \geq f\left(  x+y\right)  +f\left(  y+z\right)  +f\left(
z+x\right)  +f\left(  0\right)  \label{HH1}%
\end{equation}
for all points $x,y,z\in\lbrack0,A]$ such that $x+y+z\leq A;$ if in addition
$f(0)\geq0,$ then%
\begin{equation}
f\left(  x\right)  +f\left(  y\right)  +f\left(  z\right)  +f\left(
x+y+z\right)  \geq f\left(  x+y\right)  +f\left(  y+z\right)  +f\left(
z+x\right)  . \label{HH2}%
\end{equation}

\end{proposition}

The Hornich-Hlawka functional inequality is not characteristic to the 3-convex
functions. Indeed, as was noticed by Sendov and Zitikis \cite{SZ2014}, Theorem
4.2 (see also \cite{NS2023}, Theorem 7), the Hornich-Hlawka functional
inequality (HH2) also works in the framework of completely monotone functions.
Recall that a function $f:[0,\infty)\rightarrow\mathbb{R}_{+}$ is
\emph{completely monotone }if it is continuous on $[0,\infty)$, indefinitely
differentiable on $(0,\infty)$ and
\[
(-1)^{n}f^{(n)}(x)\geq0\text{\quad for all }x>0\text{ and }n\geq0.
\]

Some simple examples are $e^{-x},$ $1/(1+x),~$and $\left(  1/x\right)
\log(1+x).$ Notice that every completely monotone function is nonincreasing,
convex and 3-concave.

The result of Proposition \ref{prop3conv} can be considerably improved by
adding additional hypothesis.

\begin{theorem}
\label{thmRes_gen}Suppose that $f:[0,A]\rightarrow\mathbb{R}$ is a continuous
$3$-convex function which is also nondecreasing and concave. Then
\[
f\left(  \left\vert x\right\vert \right)  +f\left(  \left\vert y\right\vert
\right)  +f\left(  \left\vert z\right\vert \right)  +f\left(  \left\vert
x+y+z\right\vert \right)  \geq f\left(  \left\vert x+y\right\vert \right)
+f\left(  \left\vert y+z\right\vert \right)  +f\left(  \left\vert
z+x\right\vert \right)  +f(0)
\]
for all $x,y,z\in\lbrack-A,A]$ with $\left\vert x\right\vert +\left\vert
y\right\vert +\left\vert z\right\vert \leq A.$
\end{theorem}

Theorem \ref{thmRes_gen} is implicit in a paper due to Ressel, who formulated
his result in terms of differences of higher order. See \cite{Res},Theorem 1
and formula (5). For the convenience of the reader we will include here a full argument.

Notice first that the inequality stated in Theorem \ref{thmRes_gen} is
invariant under the permutations of the elements $x,y,z$ and also to the
symmetry $(x,y,z)\rightarrow(-x,-y,-z).$ As a consequence, the proof of
Theorem \ref{thmRes_gen} can be reduced to the following two cases:

\medskip{}

\textbf{Case }$\mathbf{1}$: the elements $x,y$ and $z$ have the same sign, in
which case we may reduce ourselves to the situation where
\[
x\geq y\geq z\geq0;
\]

\textbf{Case }$\mathbf{2}$: two of the elements $x,y,z$ are nonnegative, while
the third is nonpositive, in which case the proof reduces to the situation
where
\[
x\geq y\geq0\geq z.
\]

Case $1$ is covered by the assertion $(iii)$ of Theorem \ref{thm3conv}. Case
$2$ can be split into four subcases:

\medskip{}

\textbf{Case }$\mathbf{2a}$\textbf{: }$x\geq y\geq0\geq z$ and $\left\vert
z\right\vert \geq x+y;$

\textbf{Case }$\mathbf{2b}$\textbf{: }$x\geq y\geq0\geq z$ and $x+y\geq
\left\vert z\right\vert \geq x,y;$

\textbf{Case }$\mathbf{2c}$\textbf{: }$x\geq y\geq0\geq z$ and $x\geq
\left\vert z\right\vert \geq y;$

\textbf{Case }$\mathbf{2d}$\textbf{: } $z\leq0\leq y\leq x$ and $x\geq
y\geq\left\vert z\right\vert .$

\medskip{}

The assertion of Theorem \ref{thmRes_gen} in Case $2a$ makes the objective of
Lemma \ref{lem4a}, while the other cases (Case $2b$, Case $2c$ and Case $2d$)
are settled by Lemma \ref{lem4bcd}.

\begin{lemma}
\label{lem4a}Suppose that $f:[0,A]\rightarrow\mathbb{R}$ is a continuous
function such that
\[
f\left(  x\right)  +f\left(  y\right)  +f\left(  z\right)  +f\left(
x+y+z\right)  \geq f\left(  x+y\right)  +f\left(  y+z\right)  +f\left(
z+x\right)  +f(0)
\]
for all $x,y,z\geq0$ with $x+y+z\leq A$. Then $f$ also verifies the functional
inequality
\[
f\left(  \left\vert x\right\vert \right)  +f\left(  \left\vert y\right\vert
\right)  +f\left(  \left\vert z\right\vert \right)  +f\left(  \left\vert
x+y+z\right\vert \right)  \geq f\left(  \left\vert x+y\right\vert \right)
+f\left(  \left\vert y+z\right\vert \right)  +f\left(  \left\vert
z+x\right\vert \right)  +f(0)
\]
for all triplets $x,y,z\in\lbrack-A,A]$ of which two elements are nonnegative
and their sum does not exceed the absolute value of the third element.
\end{lemma}

\begin{proof}
It suffices to consider the case where $x,y\geq0\geq z$ and $x+y\leq\left\vert
z\right\vert $. Then, $\left\vert z\right\vert -x=\left\vert x+z\right\vert ,$
$\left\vert z\right\vert -y=\left\vert y+z\right\vert $ and $\left\vert
z\right\vert -x-y=\left\vert x+y+z\right\vert $. According to our hypothesis,
applied to $x,~y$ and $\left\vert z\right\vert -x-y,$ we have
\[
f(x)+f(y)+f\left(  \left\vert z\right\vert -x-y\right)  +f\left(  \left\vert
z\right\vert \right)  \geq f(x+y)+f\left(  \left\vert z\right\vert -y\right)
+f\left(  \left\vert z\right\vert -x\right)  +f(0),
\]
equivalently,
\begin{multline*}
f\left(  \left\vert x\right\vert \right)  +f\left(  \left\vert y\right\vert
\right)  +f\left(  \left\vert z\right\vert \right)  +f\left(  \left\vert
x+y+z\right\vert \right) \\
\geq f\left(  \left\vert y+z\right\vert \right)  +f\left(  \left\vert
x+z\right\vert \right)  +f\left(  \left\vert x+y\right\vert \right)  +f(0).
\end{multline*}

\end{proof}

\begin{lemma}
\label{lem4bcd}If $f:[0,A]\rightarrow\mathbb{R}$ is a nondecreasing and
concave function, then
\[
f\left(  \left\vert x\right\vert \right)  +f\left(  \left\vert y\right\vert
\right)  +f\left(  \left\vert z\right\vert \right)  +f\left(  \left\vert
x+y+z\right\vert \right)  \geq f\left(  \left\vert x+y\right\vert \right)
+f\left(  \left\vert y+z\right\vert \right)  +f\left(  \left\vert
z+x\right\vert \right)  +f\left(  0\right)  ,
\]
for all $x,y,z\in\lbrack-A,A]$ such that $z\leq0\leq y\leq x.$
\end{lemma}

When the domain of $f$ is $\mathbb{R}_{+}$ and $f$ is nonnegative, then the
property of concavity implies the property of being nondecreasing. See
\cite{NP2018}, Exercise 4, p. 31.

\begin{proof}
The range $z\leq0\leq y\leq x$ can be split into the following cases:

Case $2b:$ $z\leq0\leq y\leq x$ and $x,y\leq\left\vert z\right\vert \leq x+y.$
Then $0\leq x+y+z\leq\left\vert z\right\vert \leq x+y,$ so by Lemma
\ref{lem_maj} it follows that
\[
f(x+y)+f(0)\leq f(\left\vert z\right\vert )+f(x+y+z).
\]
On the other hand, $\left\vert x+z\right\vert =\left\vert z\right\vert -x,$
$\left\vert y+z\right\vert =\left\vert z\right\vert -y$ and $\left\vert
x+y+z\right\vert =x+y+z.$ Since $f$ is nondecreasing, we have $f(\left\vert
z\right\vert -x)\leq f(y)$ and $f(\left\vert z\right\vert -y)\leq f(\left\vert
x\right\vert )$ . Therefore
\begin{multline*}
f\left(  \left\vert x+y\right\vert \right)  +f\left(  \left\vert
y+z\right\vert \right)  +f\left(  \left\vert z+x\right\vert \right)  +f(0)\\
=f(x+y)+f(0)+f(\left\vert z\right\vert -y)+f(\left\vert z\right\vert -x)\\
\leq f(x+y+z)+f(\left\vert z\right\vert )+f(x)+f(y)\\
=f(\left\vert z\right\vert )+f(\left\vert x+y+z\right\vert )+f(\left\vert
x\right\vert )+f(\left\vert y\right\vert )
\end{multline*}

Case $2c:$ $z\leq0\leq y\leq x$ and $y\leq\left\vert z\right\vert \leq x.$
Applying Lemma \ref{lem_maj} for $0\leq y\leq x\leq x+y$ and taking into
account that $f$ is nondecreasing we obtain
\begin{multline*}
f\left(  \left\vert x+y\right\vert \right)  +f\left(  \left\vert
y+z\right\vert \right)  +f\left(  \left\vert z+x\right\vert \right)  +f(0)\\
=f(x+y)+f(0)+f(\left\vert z\right\vert -y)+f(x-\left\vert z\right\vert )\\
\leq f(x)+f(y)+f(\left\vert z\right\vert )+f(x+y+z)\\
=f(\left\vert x\right\vert )+f(\left\vert y\right\vert )+f(\left\vert
z\right\vert )+f(\left\vert x+y+z\right\vert ).
\end{multline*}

Case $2d:$ $z\leq0\leq y\leq x$ and $\left\vert z\right\vert \leq y\leq x.$
Similar to Case $2c$. Applying Lemma \ref{lem_maj} for $0\leq\left\vert
z\right\vert \leq x+y-\left\vert z\right\vert \leq x+y$ and using the fact
that $f$ is nondecreasing we obtain
\begin{multline*}
f\left(  \left\vert x+y\right\vert \right)  +f\left(  \left\vert
y+z\right\vert \right)  +f\left(  \left\vert z+x\right\vert \right)  +f(0)\\
=f(x+y)+f(0)+f(y-\left\vert z\right\vert )+f(x-\left\vert z\right\vert )\\
\leq f\left(  x+y-\left\vert z\right\vert \right)  +f\left(  \left\vert
z\right\vert \right)  +f(y-\left\vert z\right\vert )+f(x-\left\vert
z\right\vert )\\
\leq f(x+y+z)+f(\left\vert z\right\vert )+f(y)+f(x)\\
=f(\left\vert x\right\vert )+f(\left\vert y\right\vert )+f(\left\vert
z\right\vert )+f(\left\vert x+y+z\right\vert ).
\end{multline*}

The proof of Lemma \ref{lem4bcd} is now complete.
\end{proof}

Lemma \ref{lem4bcd}, fails in the case of nondecreasing and $3$-convex
functions which are not concave. To check this, consider the restriction of
the cubic function $x^{3}$ to $[0,\infty)$ and the triplet $x=y=1$ and $z=-1.$

Some example illustrating Theorem \ref{thmRes_gen} and Corollary
\ref{cor_comp} in the case of the Bernstein functions $x/(1+x)$, $\log(1+x)$
and the identity of $[0,\infty))$ are indicated below:

\begin{enumerate}
\item[$(RHH)$] the \emph{rational form of the Hornich-Hlawka inequality,}
\begin{multline*}
\frac{\left\vert x\right\vert ^{\alpha}}{1+\left\vert x\right\vert ^{\alpha}%
}+\frac{\left\vert y\right\vert ^{\alpha}}{1+\left\vert y\right\vert ^{\alpha
}}+\frac{\left\vert z\right\vert ^{\alpha}}{1+\left\vert z\right\vert
^{\alpha}}+\frac{\left\vert x+y+z\right\vert ^{\alpha}}{1+\left\vert
x+y+z\right\vert ^{a}}\\
\geq\frac{\left\vert x+y\right\vert ^{a}}{1+\left\vert x+y\right\vert ^{a}%
}+\frac{\left\vert y+z\right\vert ^{a}}{1+\left\vert y+z\right\vert ^{a}%
}+\frac{\left\vert z+x\right\vert ^{a}}{1+\left\vert z+x\right\vert ^{a}};
\end{multline*}

\item[$(MHH)$] the \emph{multiplicative form of the Hornich-Hlawka
inequality,}
\begin{multline*}
(1+\left\vert x\right\vert )(1+\left\vert y\right\vert )(1+\left\vert
z\right\vert )(1+\left\vert x+y+z\right\vert )\allowbreak\\
\geq(1+\left\vert x+y\right\vert )(1+\left\vert y+z\right\vert )(1+\left\vert
z+x\right\vert );
\end{multline*}

\item[$(HH^{\alpha})$] the\emph{ fractional power form of the Hornich-Hlawka
inequality}:
\[
\left\vert x\right\vert ^{\alpha}+\left\vert y\right\vert ^{\alpha}+\left\vert
z\right\vert ^{\alpha}+\left\vert x+y+z\right\vert ^{\alpha}\geq\left\vert
x+y\right\vert ^{\alpha}+\left\vert y+z\right\vert ^{\alpha}+\left\vert
z+x\right\vert ^{\alpha}.
\]

Here $\alpha\in(0,1]$ is a parameter.
\end{enumerate}

The natural analogue of Theorem \ref{thmRes_gen}, for more that 3 numbers does
not hold. Indeed, according to a comment made by Freudenthal in connection
with the Hornich-Hlawka inequality, the function
\[
\sum\nolimits_{i=1}^{4}\left\vert x_{i}\right\vert -\sum\limits_{1\leq
i<j\leq4}\left\vert x_{i}+x_{j}\right\vert +\sum\limits_{1\leq i<j<k\leq
4}\left\vert x_{i}+x_{j}+x_{k}\right\vert -\left\vert x_{1}+x_{2}+x_{3}%
+x_{4}\right\vert
\]
takes both positive and negative values as the variables $x_{1},...,x_{4}$ run
over $\mathbb{R}.$

However, an extension of Theorem \ref{thmRes_gen}\ to the case\ of $n>3$
variables is still possible by using an inductive scheme due to Vasi\'{c} and
Adamovi\'{c} \cite{VA1968}. We state here a slightly modified version of their
result as appeared in \cite{MPF}, Theorem 2, p. 528:

\begin{theorem}
\label{thmVA}Suppose that $\varphi$ is a real-valued function defined on a
commutative additive semigroup $\mathcal{S}$ such that
\[
\sum\limits_{k=1}^{3}{\varphi\left(  x_{k}\right)  }+{\varphi}\left(
{\sum\limits_{k=1}^{3}}x_{k}\right)  \gtrless\sum\limits_{1\leq{i}<{j}\leq
3}{\varphi\left(  {x{_{i}+x}}_{j}\right)  }%
\]
for all $x_{1},x_{2},x_{3}\in\mathcal{S}.$ Then for each pair $\left\{
k,n\right\}  $ of integers with $2\leq k<n$ we also have%
\[
\binom{n-2}{k-1}\sum\limits_{k=1}^{n}{\varphi\left(  x_{k}\right)  }%
+\binom{n-2}{k-2}{\varphi}\left(  {\sum\limits_{k=1}^{n}}x_{k}\right)
\gtrless\sum\limits_{1\leq{i_{1}}<...<{i_{k}}\leq n}{\varphi\left(
{\sum\limits_{j=1}^{k}{x{_{i_{j}}}}}\right)  ,}%
\]
whenever $x_{1},...,x_{n}\in\mathcal{S}.$
\end{theorem}

This result yields the following generalization of Theorem \ref{thmRes_gen}:

\begin{theorem}
\label{thm_nvar}If $f:\mathbb{R}_{+}\rightarrow E$ is a continuous $3$-convex
function which is also nondecreasing and concave, then
\begin{multline*}
\binom{n-2}{k-1}\sum\limits_{k=1}^{n}{f\left(  \left\vert x_{k}\right\vert
\right)  }+\binom{n-2}{k-2}{f}\left(  \left\vert {\sum\limits_{k=1}^{n}}%
x_{k}\right\vert \right) \\
\geq\sum\limits_{1\leq{i_{1}}<...<{i_{k}}\leq n}{f\left(  \left\vert
{\sum\limits_{j=1}^{k}{x{_{i_{j}}}}}\right\vert \right)  +}\binom{n-1}{k}f(0),
\end{multline*}
for all pairs $\left\{  k,n\right\}  $ of integers with $2\leq k<n$ and all
strings $x_{1},...,x_{n}$ of real numbers.
\end{theorem}

\begin{proof}
Apply Theorem \ref{thmVA} to $f(\left\vert \cdot\right\vert )-f(0),$ taking
into account the last assertion of \ref{thmRes_gen} and the formula
\[
\binom{n-2}{k-1}+\binom{n-2}{k-2}=\binom{n-1}{k-1}.
\]

\end{proof}

\section{The case of vector-valued functions}

The concept of $n$-convexity can be extended in a straightforward way to the
case of functions with values in an ordered Banach space by using the same
definition based on divided differences.

Recall that an \emph{ordered Banach space} is any Banach space $E$ endowed
with the ordering $\leq$ associated to a closed convex cone $E_{+}$ via the
formula%
\[
x\leq y\ \text{if and only if }y-x\in E_{+},~
\]
such that%
\[
E=E_{+}-E_{+},\text{\quad}\left(  -E_{+}\right)  \cap E_{+}=\left\{
0\right\}  ,
\]
and%
\[
0\leq x\leq y\ \text{in}~E~\text{implies }\left\Vert x\right\Vert
\leq\left\Vert y\right\Vert .
\]
The basic facts concerning the theory of ordered Banach spaces are made
available by the book of Schaefer and Wolff \cite{SW}. See \cite{NO2020} for a
short overview centered on two important particular cases: $\mathbb{R}^{n},$
the $n$-dimensional Euclidean space endowed with the coordinate-wise ordering,
and $\operatorname*{Sym}(n,\mathbb{R)}$ the ordered Banach space of all
$n\times n$-dimensional symmetric matrices with real coefficients endowed with
the operator norm%
\[
\left\Vert A\right\Vert =\sup_{\left\Vert x\right\Vert \leq1}\left\vert
\langle Ax,x\rangle\right\vert
\]
and the Löwner ordering,%
\[
A\leq B\text{ if and only if }\langle A\mathbf{x},\mathbf{x}\rangle\leq\langle
B\mathbf{x},\mathbf{x}\rangle\text{ for all }\mathbf{x}\in\mathbb{R}^{n}.
\]

Here the operator norm \ can be replaced by any Schatten norm, in particular
with the \emph{Frobenius norm},
\[
\left\Vert A\right\Vert _{F}=\left(  \sum_{i=1}^{n}\sum_{j=1}^{n}a_{ij}%
^{2}\right)  ^{1/2},
\]
provided that $A=\left(  a_{ij}\right)  _{i,j=1}^{n}.$ The Frobenius norm is
associated to the trace inner product%
\[
\langle A,B\rangle=\operatorname*{trace}(AB).
\]

The positive cone of $\mathbb{R}^{n}$ is the first orthant $\mathbb{R}_{+}%
^{n},$ while the positive cone of $\operatorname*{Sym}(n,\mathbb{R)}$ is the
set $\operatorname*{Sym}^{+}(n,\mathbb{R)}$ consisting of all positive
semi-definite matrices.

\begin{remark}
\label{rem_order}The study of vector-valued functions can be reduced to that
of real-valued functions.\ Indeed, in any ordered Banach space $E,$ any
inequality of the form$\ u\leq v$ is equivalent to $x^{\ast}(u)\leq x^{\ast
}(v)$ for all $x^{\ast}\in E_{+}^{\ast}$. See \emph{\cite{NO2020}}.

As a consequence, a function $f:I\rightarrow E$ is respectively nondecreasing,
convex or $n$-convex if and only if $x^{\ast}\circ f$ has this property
whenever $x^{\ast}\in E^{\ast}$ is a positive functional. For $E=\mathbb{R}%
^{n},$ this remark concerns the components of $f$.
\end{remark}

Remark \ref{rem_order} easily yields that most of the results in the preceding
sections extends verbatim to the vector-valued framework. In particular, so
are Theorem \ref{thm3conv}, Proposition \ref{prop3conv}, Theorem
\ref{thmRes_gen}, Theorem \ref{thm3HH} and Theorem \ref{thm_nvar}.

Combining Remark \ref{rem_order} with Lemma \ref{lemH} one obtains the
following practical test of $3$-convexity for the vector-valued differentiable functions:

\begin{theorem}
\label{lem_vector_3conv}Suppose that $f$ is a continuous function defined on
an interval $I$ and taking values in an ordered Banach space $E$. If $f$ is
three times differentiable on the interior of $I,$ then $f$ is a $3$-convex
function if and only if $f^{\prime\prime\prime}\geq0.$
\end{theorem}

An example illustrating Theorem \ref{lem_vector_3conv} is provided by the
function
\[
f:\mathbb{R}_{+}\rightarrow\operatorname*{Sym}(n,\mathbb{R)},\text{\quad
}f(t)=-e^{-tA},
\]
associated to a positive semi-definite matrix $A\in\operatorname*{Sym}%
(n,\mathbb{R)}.$ This function is of class $C^{\infty}$ and its first three
derivatives are given by the formulas
\[
f^{\prime}(t)=Ae^{-tA},\text{\quad}f^{\prime\prime}(t)=-A^{2}e^{-tA}%
,\text{\quad}f^{\prime\prime\prime}(t)=A^{3}e^{-tA}.
\]
This shows that $f$ is nondecreasing, concave and 3-convex (according to the
ordering of $\operatorname*{Sym}(n,\mathbb{R))}.$ The fact that $A^{3}e^{-tA}$
is positive semidefinite follows from the fact that the product of positive
semi-definite matrices that commute with each other is a matrix of the same type.

According to Theorem \ref{thmRes_gen},
\[
e^{-\left\vert r\right\vert A}+e^{-\left\vert s\right\vert A}+e^{-\left\vert
t\right\vert A}+e^{-\left\vert r+s+t\right\vert A}\leq I+e^{-\left\vert
r+s\right\vert A}+e^{-\left\vert s+t\right\vert A}+e^{-\left\vert
t+r\right\vert A}%
\]
for all $r,s,t\in\mathbb{R}.$ Here $I$ is the identity matrix. In the
1-dimensional case, this reduces to the inequality
\[
e^{-\alpha\left\vert x\right\vert }+e^{-\alpha\left\vert y\right\vert
}+e^{-\alpha\left\vert z\right\vert }+e^{-\alpha\left\vert x+y+z\right\vert
}\leq1+e^{-\alpha\left\vert x+y\right\vert }+e^{-\alpha\left\vert
y+z\right\vert }+e^{-\alpha\left\vert x+z\right\vert },
\]
which works for all $x,y,z\in\mathbb{R}$ and $\alpha>0$. This last inequality
can be extended to the framework of real symmetric matrices:

\begin{theorem}
\label{thm_funct}Suppose that $f:[0,\infty)\rightarrow\mathbb{R}$ is a
continuous $3$-convex function which is also nondecreasing and concave. Then
\begin{multline*}
f\left(  \left\vert A\right\vert \right)  +f\left(  \left\vert B\right\vert
\right)  +f\left(  \left\vert C\right\vert \right)  +f\left(  \left\vert
A+B+C\right\vert \right) \\
\geq f\left(  \left\vert A+B\right\vert \right)  +f\left(  \left\vert
B+C\right\vert \right)  +f\left(  \left\vert C+A\right\vert \right)
+f(0)I_{n},
\end{multline*}
whenever $A,B,C$ are three real symmetric matrices of order $n$ that commute
with each other.
\end{theorem}

Here $\left\vert A\right\vert =\left(  A^{2}\right)  ^{1/2}$ denotes the
modulus of $A$ and $I_{n}$ is the unit matrix of order $n$.

\begin{proof}
Notice first that every finite family of self-adjoint matrices that commute
with each other admits an orthonormal basis consisting of vectors that are
eigenvectors of each these matrices. See Mirsky \cite{Mir}, Theorem 10.6.8, p.
322. This reduces the proof of the theorem to the case where all the matrices
$A,$ $B$ and $C$ are diagonal. Or, if
\[
A=\left(
\begin{array}
[c]{ccc}%
\lambda_{1}(A) &  & 0\\
& \ddots & \\
0 &  & \lambda_{n}(A)
\end{array}
\right)  ,
\]
then%
\[
f(A)=\left(
\begin{array}
[c]{ccc}%
f\left(  \lambda_{1}(A)\right)  &  & 0\\
& \ddots & \\
0 &  & f\left(  \lambda_{n}(A)\right)
\end{array}
\right)  ,
\]
so that the conclusion of the theorem follows from Theorem \ref{thmRes_gen}.
\end{proof}

\begin{remark}
Theorem \emph{\ref{thm_funct}} also works in the context of commuting
self-adjoint compact operators defined on an infinite dimensional Hilbert
space provided that $f:[0,\infty)\rightarrow\mathbb{R}$ is continuous,
nondecreasing, concave $3$-convex and $f(0)=0.$ We do not know whether the
commutativity condition is necessary or not.
\end{remark}

\end{document}